\documentclass[preprint,12pt,authoryear]{elsarticle}

\usepackage{amssymb}
\usepackage{amsthm}
\usepackage{amsmath}

\newtheorem{theorem}{Theorem}

\newtheorem{corollary}{Corollary}

\begin{document}

\begin{frontmatter}

\title{On the Inversion of Polynomials of Discrete Laplace Matrices}

\author[label1,label2]{S. Asghar\corref{cor1}}
\ead{sabia.asghar@uhasselt.be}
\author[label3,label4]{Q. Peng}
\author[label1,label2,label5,label6]{F.J. Vermolen}
\author[label5]{C. Vuik}
\cortext[cor1]{Corresponding author}
\affiliation[label1]{organization={Computational Mathematics Group (CMAT), Department of Mathematics and Statistics, University of Hasselt},addressline={Diepenbeek}, country={Belgium}}
\affiliation[label2]{organization={Data Science Institute (DSI), University of Hasselt},addressline={Diepenbeek}, country={Belgium}}
\affiliation[label3]{organization={Mathematics for AI in Real-world Systems, School of Mathematical Sciences, Lancaster University}, addressline={Lancaster}, country={UK}}
\affiliation[label4]{organization={Mathematical Institute, Leiden University}, addressline={Einsteinweg 55, 2333 CC Leiden}, country={The Netherlands}}
\affiliation[label5]{organization={Delft Institute of Applied Mathematics, Delft University of Technology}, addressline={Delft}, country={The Netherlands}}
\affiliation[label6]{organization={Department of Mathematics and Applied Mathematics, University of Johannesburg}, addressline={Johannesburg}, country={South Africa}}

\begin{abstract}
The efficient inversion of matrix polynomials is a critical challenge in computational mathematics. We design a procedure to determine the inverse of matrices polynomial of multidimensional Laplace matrices. The method is based on eigenvector and eigenvalue expansions. The method is consistent with previous expressions of the inverse discretized Laplacian in one spatial dimension \citep{Vermolen_2022}. Several examples are given.
\end{abstract}

\end{frontmatter}

\section{Introduction}
The inversion of discretized Laplace inspired matrices is often a crucial, but also rate-determining step in many simulation packages \citep{Yousef_2003, Strang_2007, Leon_2013}. Such simulations may come from diffusion, heat distribution, fluid dynamics and many other modeling problems from science and technology. Even from financial mathematics or filtering in statistics and data science, such problems are important to solve \citep{Brigo_1998, Regulwar_2023}. Developing efficient solvers and pre-conditioners for these matrices remains a core focus in numerical computing to improve simulation efficiency and accuracy. 

\par The current paper addresses a solution to the problem
\begin{align*}
A {\underline x} = {\underline b},
\end{align*}
where $A \in \mathbb{R}^{n \times n}$, where $n \in \mathbb{N}$ is relatively large. Solutions of large linear algebraic systems of equations often proceed by the use of Krylov subspace methods, multigrid methods or combinations of the two. In the case of very large systems, parallelized algorithms may be beneficial. Such systems often result from the discretization of a partial differential equations (PDEs) based on a Laplacian operator, such as 
\begin{align*}
- \Delta u = - \nabla \cdot (\nabla u) = f({\bf x}), \text{ in } \Omega,
\end{align*}
where $\Omega \subsetneq \mathbb{R}^d$ is an open, connected, bounded domain, and appropriate boundary conditions are assigned to $u$ that warrant existence and uniqueness of the solution. In these cases, the discrete counterpart of the differential operator, being the matrix $A$, should be invertible if an adequate discretization (in terms of consistency, stability and hence convergence) is used.
Some other PDEs, such as the biharmonic equation or the Cahn-Hilliard equation, contain (linear) combinations of powers of Laplace operators. We will consider such types of equations. 
The resolution of these types of equations is of significant interest due to its broad applicability across multiple disciplines, including physics, chemistry, mathematics, and economics \citep{Gurarslan}. Their general forms encompass several fundamental classes of partial differential equations (PDEs), each governing critical phenomena:
\begin{itemize}
\item Elliptic Equations: The Poisson equation, which models electrostatic and magnetostatic fields, incompressible fluid flow (both inviscid and viscous), vortex dynamics, and fluid filtration through porous media. Somewhat more complicated, but in the same philosophy, we mention (linear) elasticity as an important application in mechanics, and the Euler-Tricomi equation, which is essential for studying transonic flow dynamics.
\item Parabolic Equations: The extended diffusion equation and the Diffusion-Advection-Reaction (DAR) equation, which describe processes such as heat transfer, mass transport, and electromagnetic field propagation.
The Black-Scholes equation, widely used in mathematical finance for option pricing \citep{Crank}.
\item Hyperbolic Equations: The extended wave equation and telegraph equation, which play a crucial role in electromagnetics and telecommunications \citep{Mittal}. Other examples concerns Buckley-Leverett or Burgers' equations that arise in applications in modeling flow in porous media.
\end{itemize}
Given its wide-ranging implications, we propose a simple, accurate, and possibly computationally efficient numerical approach that aims to provide an exact solution while maintaining high precision, striking an optimal balance between simplicity and accuracy. Furthermore, the approach that we use to obtain closed form solutions to a class of systems of linear equations, as well as inverses of classes of matrices is of theoretical value.
Convection-based operators, or equations based on odd-order spatial derivatives, as well as nonlinear problems, will not be considered in this paper. Furthermore, we will limit ourselves to linear combinations of powers of Laplace operators. We will develop a closed-form expression for the solution $\underline{x}$, as well as a formalization to express $A^{-1}$. One usually does not determine the inverse of a matrix, however, our method allows to do so for theoretical purposes. The analysis and approach will be conceptually remarkably simple since it is based on Von Neumann analysis, which we will not use for the assessment of stability, but for the sake of determining eigenvalues of $A$.

The paper deals with the important class of discrete Laplace matrices. In Section 2, we introduce a numerically efficient method based on eigenvector expansions and derive the underlying inversion principle in Section 3. The practical utility and performance of our technique are then validated through a series of case studies in Section 4, followed by discussion and conclusions in Section 5.
{\section{The eigenvector expansion technique}
Consider the linear system of equations
\begin{equation}
    A \underline{x} = \underline{b},
    \label{eq1}
\end{equation} 
where $A$ is a nonsingular, symmetric $n \times n$ matrix and $\underline{x}$ and $\underline{b}$ are vectors in $\mathbb{R}^n$. The matrix $A$ has $n$ orthogonal eigenvectors by the Principal Axis Theorem \citep{Gilbert} for symmetric matrices, which form a basis for $\mathbb{R}^n$. In particular, there is a set of $n$ eigenvectors $\{ \underline{v}_k \}$, and a corresponding set of real eigenvalues $\{ \lambda_k \}$, such that
\begin{equation}
    A \underline{v}_k = \lambda_k \underline{v}_k, \quad k = 1, 2, \dots, n
\label{eq2}
\end{equation} 
and
\begin{equation*}
    (\underline{v}_i, \underline{v}_j ) = 0 \quad \text{if } i \neq j,
\end{equation*}
where $\left(\cdot, \cdot \right)$ is the standard vector inner (dot) product
\begin{equation*}
    (\underline{x}, \underline{y} ) = \frac{1}{n} \sum_{k=1}^n x_k y_k.
\end{equation*}
Since the set of vectors $\{ \underline{v}_k \}$ forms a basis, one can express $\underline{b}$ as a linear combination of those vectors
\begin{equation*}
    \underline{b} = \sum_{k=1}^n \beta_k \underline{v}_k. 
\end{equation*}
As the set of vectors is orthogonal, one can determine the coefficients $\beta_k$ by forming inner products of Equation \eqref{eq2} with each of the elements of the set $\{ \underline{v}_k \}$. For example, to obtain an expression for $\beta_j$, we form the inner product of Equation \eqref{eq2} with $\underline{v}_j$ and simplify
\begin{equation*}
    (\underline{b}, \underline{v}_j ) = \left( \sum_{k=1}^n \beta_k \underline{v}_k, \underline{v}_j \right) = \sum_{k=1}^n \beta_k (\underline{v}_k, \underline{v}_j ) = \beta_j ( \underline{v}_j, \underline{v}_j ).
\end{equation*}
Hence,
\begin{align}
\beta_j = \frac{(\underline{b}, \underline{v}_j )}{( \underline{v}_j, \underline{v}_j )}.
\label{beta-eq}
\end{align}
Note that we did not yet normalize the eigenvectors of $A$.
We can express the solution $\underline{x}$ of \eqref{eq1} as a linear combination of the vectors $ \{\underline{v}_k\}$ 
\begin{equation*}
    \underline{x} = \sum_{k=1}^n c_k \underline{v}_k.
\end{equation*}
If $\underline{x}$ is to be a solution of $A \underline{x} = \underline{b}$, then
\begin{equation*}
    A \left( \sum_{k=1}^n c_k \underline{v}_k \right) = \underline{b}.
\end{equation*}
Using the eigenvector property of $\underline{v}_k$, one gets
\begin{equation*}
    \sum_{k=1}^n c_k A \underline{v}_k = \sum_{k=1}^n \beta_k \underline{v}_k
\end{equation*}
\begin{equation*}
    \Rightarrow \sum_{k=1}^n c_k \lambda_k \underline{v}_k = \sum_{k=1}^n \beta_k \underline{v}_k.
\end{equation*}
By forming inner products with $\underline{v}_j$, we obtain
\begin{equation*}
    \sum_{k=1}^n c_k \lambda_k(\underline{v}_k, \underline{v}_j ) = \sum_{k=1}^n \beta_k (\underline{v}_k, \underline{v}_j ).
\end{equation*}
Using orthogonality,
\begin{equation*}
    c_j \lambda_j (\underline{v}_j, \underline{v}_j ) = \beta_j ( \underline{v}_j, \underline{v}_j ),
\end{equation*}
Thus
\begin{equation*}
    c_j = \frac{\beta_j}{\lambda_j}.
\end{equation*}
Hence we conclude that, if the matrix $A$ possesses a set of $n$ orthogonal eigenvectors, then the solution to $A \underline{x} = \underline{b}$ is given by the eigenvector expansion
\begin{equation*}
    \underline{x} = \sum_{k=1}^n \frac{\beta_k}{\lambda_k} \underline{v}_k,
\end{equation*}
where $\lambda_k \ne 0$ is the eigenvalue associated with $\underline{v}_k$, and the $\beta_k$'s are determined from Equation (\ref{beta-eq}). Note that this procedure is analogous to the procedures used in separation of variables in solving partial differential equations and that we did not yet normalize the eigenvectors.
\section{The Principle for Inverses of Matrix Polynomials}
We continue with the system $A \underline{x} = \underline{b}$, where $A \in \mathbb{R}^{n \times n}$ is a symmetric positive definite matrix (for instance representing the finite difference representation of the Laplace operator with Dirichlet boundary conditions), which gives positive (real-valued) eigenvalues and orthogonal eigenvectors. Let the normalized eigenvectors of $A$ be given by ${\underline v}_k$, $k = 1,\ldots,n$, with respective eigenvalues $\lambda_k$, then we arrive at the following expression for the solution
\begin{align}
{\underline x} = \sum_{j=1}^n x_j {\underline v}_j = \sum_{j=1}^n \frac{1}{\lambda_j} (\underline{b},\underline{v}_j) \underline{v}_j.
\label{gen-sol}
\end{align}
This entirely fits within the eigenvalue expansion in the continuous case. The next step is to express the inverse $A^{-1}$ in terms of the eigenvalues and eigenvalues. In determining the inverse, one can proceed columnwisely. Let $\underline{g}_k$ be the $k$-th column of $A^{-1}$, then $\underline{g}_k$ satisfies
\begin{align}
A \underline{g}_k = \underline{e}_k = [ \ldots 1 \ldots]^T.
\end{align}
Here $\underline{e}_k$ is a vector with zeros, except for the $k$-th position where it has the value 1. Using Equation (\ref{gen-sol}), this implies that 
\begin{align}
\underline{g}_k = \sum_{j=1}^n \frac{1}{\lambda_j} (\underline{e}_k,\underline{v}_j) \underline{v}_j = 
\sum_{j=1}^n \frac{{v}_{jk}}{n \lambda_j}  ~\underline{v}_j,
\end{align}
where $v_{jk}$ represents the $k$-th component of the $j$-th eigenvector of $A$. In order to get element $(A^{-1})_{ik}$, one computes from
\begin{align}
(A^{-1})_{ik} = G_{ki} =
\sum_{j=1}^n \frac{{v}_{jk}}{n \lambda_j}  ~v_{ji},
\end{align}
where $G = [\underline{g}_1 \ldots \underline{g}_n]$ is the matrix with the $g$-vectors as its columns.
Hence, formally $A^{-1} = G^T$. Note that $A^{-1}$ is symmetrical since the inverse of a symmetric positive matrix is also symmetric positive definite, and therefore we have $A^{-1} = G$. The principle of the above equation is used in the computations that will follow in the later sections. Furthermore, we note that if $\lambda_j$ and $\underline{v}_j$ form an eigenpair of the symmetric positive definite matrix $A$, then the matrix $A^m$, $m \in \mathbb{Z}$, is also symmetric positive definite with eigenpair $\lambda^{m}_j$ and $\underline{v}_j$. Hence the solution to 
\begin{align}
A^m {\underline x} = {\underline b},
\end{align}
is given by 
\begin{align}
{\underline x} = 
\sum_{j=1}^n \frac{1}{\lambda_j^m} (\underline{b},\underline{v}_j) \underline{v}_j,
\end{align}
and the inverse of $A^m$, denoted by $A^{-m}$, is given by 
\begin{align}
(A^{-m})_{ik} = 
\sum_{j=1}^n \frac{{v}_{jk}}{n \lambda_j^m}  ~v_{ji}. 
\end{align}
This is generalized in the following assertion:
\begin{theorem}\label{thm:poly-eig}
Let $A$ be an $n \times n$ matrix over $\mathbb{R}$ with eigenvalues $\lambda_1, \lambda_2, \ldots, \lambda_n$ (counting algebraic multiplicities). For any polynomial 
\[ P(x) = a_kx^k + a_{k-1}x^{k-1} + \cdots + a_1x + a_0, \]
the matrix $P(A)$ has eigenvalues $P(\lambda_i)$ for $i = 1,\ldots,n$.
\end{theorem}
\begin{proof}
We prove this in two parts. In the first part, we consider $P(x) = x^m$. Let $\lambda$ be an eigenvalue of $A$ with eigenvector $\underline{v} \neq \underline{0}$
\[ A\underline{v} = \lambda\underline{v},\]
Then by induction
\begin{align*}
A^2{\underline{v}} &= A(A\underline{v}) = A(\lambda\underline{v}) = \lambda A\underline{v} = \lambda^2\underline{v} \\
A^3\underline{v} &= A(A^2\underline{v}) = A(\lambda^2\underline{v}) = \lambda^2 A\underline{v} = \lambda^3\underline{v} \\
&\vdots \\
A^m\underline{v} &= \lambda^m\underline{v}.
\end{align*}
Thus $\lambda^m$ is an eigenvalue of $A^m$ with the same eigenvector $\underline{v}$ \citep{Gilbert}.
\par In the second part, we consider general $P(A) = a_kA^k + \cdots + a_0I$. For any eigenpair $(\lambda, \underline{v})$ of $A$, we have
\begin{align*}
P(A)\underline{v} &= \left(\sum_{j=0}^k a_j A^j\right)\underline{v} 
= \sum_{j=0}^k a_j (A^j \underline{v}) 
= \sum_{j=0}^k a_j \lambda^j \underline{v} \quad \text{(by Part 1)} \\
&= \left(\sum_{j=0}^k a_j \lambda^j\right)\underline{v} 
= P(\lambda)\underline{v}.
\end{align*}
Hence $P(\lambda)$ is an eigenvalue of $P(A)$ with eigenvector $\underline{v}$.
If $\lambda$ has algebraic multiplicity $m$, then for the characteristic polynomial, one has
\[ \det(A - \lambda I) = (\lambda_i - \lambda)^m q(\lambda) \]
The polynomial $P(A)$ will have eigenvalue $P(\lambda)$ with at least multiplicity $m$ (may increase if $P(\lambda_i) = P(\lambda_j)$ for $i \neq j$) \citep{Horn}.
\end{proof}
\begin{corollary}[Extension to Diagonalizable Matrices]
If $A$ is diagonalizable with $A = QDQ^{-1}$, then
\[ P(A) = Q \begin{pmatrix}
P(\lambda_1) & & \\
& \ddots & \\
& & P(\lambda_n)
\end{pmatrix} Q^{-1} \]
\end{corollary}

Subsequently, knowing the eigenvalues and eigenvectors of $A$, then for the general matrix polynomial of order $m$, one obtains for each eigenpair $\lambda$ and $\underline{v}$
\begin{align}
P(A) \underline{v} = (\beta_0 I + \beta_1 A + \ldots + \beta_m A^m) {\underline v} = (\beta_0 + \beta_1 \lambda + \ldots + \beta_m \lambda^m) \underline{v},
\end{align}
which implies that this matrix polynomial has the same eigenvectors as $A$ with eigenvalue $P(\lambda) = \beta_0 + \beta_1 \lambda + \ldots + \beta_m \lambda^m$, which according to the Fundamental Theorem from Algebra, has at least one (complex) zero and the polynomial can be factorized in terms of its zeros by the Factorization Theorem. If $\beta_j > 0, \forall~j$, then this expression will never be zero since $\lambda > 0$ (recall that $A$ is symmetric positive definite). The matrix polynomial is symmetric positive definite. Therefore, for the equation
\begin{align}
P(A) {\underline x} = {\underline b},
\end{align}
the solution is expressed by
\begin{align}
{\underline x} = 
\sum_{j=1}^n \frac{1}{P(\lambda_j)} (\underline{b},\underline{v}_j) \underline{v}_j,
\end{align}
Hence for the inverse of a matrix polynomial, we get
\begin{align}
(P(A)^{-1})_{ik} = 
\sum_{j=1}^n \frac{{v}_{jk}}{n P(\lambda_j)}  ~v_{ji}. 
\end{align}
Hence we have expressed the solution to a linear system of equations and the inverse of a matrix, in case that the matrix is polynomial in a symmetric positive definite matrix, $A$, in terms of the eigenvalues and eigenvectors of the original matrix $A$. Once the eigenvalues and eigenvectors of $A$ have been determined, then it is straightforward to determine the inverse of any polynomial provided that the matrix $P(A)$ is nonsingular. Of course we note that in most cases one is not interested in the inverse of the matrix polynomial, but merely in the solution $\underline{x}$.
In the next section, we will derive some practical results.

\section{Case Studies}
\subsection{One dimensional Laplacian matrix}
We first illustrate how the method works for a one-dimensional case and we show that our result is consistent with our earlier results \citep{Vermolen_2022}. For this purpose, we consider a simple Dirichlet problem, given by
\begin{align*}
\begin{cases}
-u'' = f(x), \quad 0 < x < 1, \\
u(0) = u(1) = 0.
\end{cases}
\end{align*}
We divide the interval $(0,1)$ into equidistant mesh points, $x_j = j h$ with spacing $h$. Having $n$ unknowns, we have $(n+1) h = 1$. Note that the total number of nodal points is $n+2$ and that the number of degrees of freedom (unknowns) is given by $n$. Having Dirichlet boundary conditions, gives the following $n \times n$-matrix:
\begin{align}
A = 
\begin{pmatrix}
\frac{2}{h^2} & -\frac{1}{h^2} & \ldots & \ldots & \ldots  \\
\ldots & \ldots & \ldots & \ldots & \ldots\\
\ldots & -\frac{1}{h^2} & \frac{2}{h^2} & -\frac{1}{h^2} & \ldots \\
\ldots & \ldots & \ldots & \ldots & \ldots \\
\ldots & \ldots & \ldots & -\frac{1}{h^2} & \frac{2}{h^2}
\end{pmatrix}
\label{discrmatr1D}
\end{align}
The above matrix is symmetric positive definite. In our previous work \citep{Vermolen_2022}, the solution to the equation $A \underline{x} = \underline{b}$ was obtained through the mapping of the piecewise linear fundamental solution to 
the continuous problem to the finite difference mesh. Since the higher-order derivatives of the fundamental solution are zero, the exact solution and numerical solution are equal. This leads to an expression for $A^{-1}$. This has been done for generic meshes and boundary conditions in $\mathbb{R}^1$. Now we use an approach based on the eigenvectors and eigenvalues of this matrix $A$. The corresponding continuous eigenvalue problem becomes
\begin{align}
\begin{cases}
u'' + \mu^2 u = 0, \\
u(0) = u(1) = 0.
\end{cases}
\end{align}
The eigenvectors and (normalized) eigenfunctions of this Sturm-Liouville problem are given by
\begin{align}
\hat{\lambda}_j = \mu_j^2 = j^2 \pi^2, \qquad \phi_j(x) = \sqrt{2} \sin(j \pi x), \qquad j = 1,2,\ldots.
\end{align}
As an Ansatz for the eigenvectors of $A$, we project the eigenfunction on the finite difference mesh $\{x_j\}_{j=1}^n$, with $x_j = j h$ and $h = \frac{1}{n+1}$. Then for the $i$-th row multiplication of $A$, one obtains
\begin{align*}
\sum_{j = 1}^n a_{ij} \phi_k(x_j) 
&= \frac{1}{h^2} (-\phi_k(x_j-h) + 2 \phi_k(x_j) - \phi_k(x_j + h))\\
&=  \frac{2\sqrt{2}}{h^2} (1 - \cos(k \pi h)) \sin(k \pi x_j) = \lambda_k \phi_k(x_j).\\
\end{align*}
Hence the projection of the eigenfunctions of the continuous problem onto the finite difference mesh points gives the eigenvectors of the discretization matrix. Furthermore, the eigenvalues of the discretization matrix are given by 
\begin{align}
\lambda_k = \frac{2}{h^2} (1 - \cos(k \pi h)) = \frac{4}{h^2} \sin^2(\frac{k \pi h}{2}) \in (0,\frac{4}{h^2}), \text{ for } k = 1,\ldots, n,
\end{align}
with corresponding eigenvector
\begin{align}
v_{kj} = \sqrt{2} \sin(k \pi x_j) = \sqrt{2} \sin(k \pi j h) = \phi_k(x_j) = \phi_k(jh), \qquad j = 1, \ldots, n.
\end{align}
Hence the eigenvectors are indeed constructed from the projection of the eigenfunctions of the continuous Laplace operator onto the finite difference mesh. Then the eigenvalues and eigenvectors are available and can be used to compute the solution to $A \underline{x} = \underline{b}$ and the inverse $A^{-1}$, as well as the inverse of a matrix polynomial. Using Equation (\ref{gen-sol}) for the current one dimensional Laplace operator with Dirichlet boundary conditions, the solution of $A \underline{x} = \underline{b}$ is formally given by
\begin{align}
\underline{x} = \sum_{j=1}^n \frac{h^2 (\underline{b},\underline{v}_j)}{4 \sin^2(\frac{j \pi h}{2})} \underline{v}_j = \frac{h^2}{n+1}
\sum_{j=1}^n \displaystyle{\frac{\sum_{k=1}^{n+1} (b_k v_{jk})}{4 \sin^2(\frac{j \pi h}{2})} \underline{v}_j},
\end{align}
where $v_{jk} = \sqrt{2} \sin(j \pi k h)$ has been chosen. Further for the inner product, $n+1$ has been chosen since the number of intervals $x_j-x_{j-1} = h$ is $n+1$. This gives consistency with 'continuous inner product' based on an integral. Note further that $v_{j,n+1} = 0$, which implies that the $n+1$-th term does not have any influence and the only adaptation is the division by $n+1$ instead of $n$ for the sake of consistency with the $n+1$ intervals. Hence the inverse of $A$ is expressed by
\begin{align}
(A^{-1})_{ik} = G_{ki} =
\frac{h^2}{4(n+1)} \sum_{j=1}^n \frac{{v}_{jk}}{ \sin^2(\frac{j \pi h}{2})}  ~v_{ji}
=
\frac{h^2}{2(n+1)} \sum_{j=1}^n \frac{\sin(j \pi k h)}{ \sin^2(\frac{j \pi h}{2})}  ~ \sin(j \pi i h).
\end{align}
This principle can be used for different boundary conditions as long as we do not have Neumann conditions on both boundaries because of singularity of $A$ (then $A^{-1}$ does not exist at all, so it becomes pointless). However, double Neumann conditions can be used for matrix polynomials. For the one-dimensional case, this new approach is of hardly any value compared to the fundamental solution approach which is useful since the fundamental solutions are piecewise linear in $\mathbb{R}^1$, which makes the truncation errors zero under the use of finite differences or finite elements as a result of the second derivative being zero between adjacent meshpoints.
\par For the 1D Laplace equation with Dirichlet boundary conditions, the inverse matrix is given by 
\begin{align}
(A^{-1})_{ik} = h^2 \left(\frac{n+1-k}{n+1} i - (i-k)_+\right),
\label{1D-inverse}
\end{align}
which is a much simpler expression (see for instance \citep{Vermolen_2022}).
However, for higher dimensionality, the fundamental solution is typically a logarithmic function, which has global nonzero higher-order partial derivatives, and hence the truncation error will not vanish. The same holds for polynomials of Laplace operators, even in 1D, such as the relatively simple case $-u'' + u$. For this reason, the fundamental solutions are no longer useful for both higher-dimensional problems and problems, which entail polynomials of the Laplace operator. It is easy to verify, by means of implementing both formulas in the computer, that these different expressions are equal. However, it is more elegant to verify this rigorously from a mathematical point of view using discrete Fourier transforms. We will summarize the result as a theorem and prove the equality.
\begin{theorem}
    Let $A$ be given by Equation (\ref{discrmatr1D}), then its inverse is given by
    \begin{align*}
    (A^{-1})_{ik} = h^2 \left(\frac{n+1-k}{n+1}~i - (i-k)_+\right) = \frac{h^2}{2(n+1)} \sum_{j=1}^n \frac{\sin(j \pi k h)}{ \sin^2(\frac{j \pi h}{2})}  ~ \sin(j \pi i h).
    \end{align*}
\end{theorem}
\begin{proof}
We remark that we just computed the right-hand side and that the first equality was determined in \citep{Vermolen_2022}. Hence, these expressions should give the same value. However, because these expressions look very different, we will algebraically demonstrate that they are consistent. Since the functions $\phi_k(x) = \sqrt{2} \sin(k \pi x)$ are eigenfunctions and eigenvectors (using $v_j = \phi_k(x_j)$) of the Laplace and Laplace matrix, respectively. Since the Laplace matrix is symmetric, the eigenvectors are orthogonal (in the continuous case, this holds as well, of course). Hence, the functions satisfy $$\sum_{j=1}^n \phi_k(x_j) \phi_l(x_j) = 0, ~ \text{ if } k \ne l.$$ We can normalize the functions by choosing $\alpha$ in $\phi_k(x_j) = \alpha ~\sin(k\pi j h)$ so that 
$$\sum_{j=1}^n \alpha^2 \sin^2(k \pi j h) = 1.$$
To this extent, we compute 
\begin{align*}
\sum_{j=1}^n \sin^2(k \pi j h) 
&= \sum_{j=1}^n \left( \frac12 - \frac12 \cos(2 k \pi j h) \right) \\
&= \frac{n}{2} - \frac14\sum_{j=1}^n \left( e^{2k\pi j h i} + e^{-2k\pi j h i} \right) \\
&= \frac{n}{2} - \frac14 \sum_{j=1}^n (e^{2 k \pi h i})^j - \frac14 \sum_{j=1}^n (e^{-2 k \pi h i})^j \\
&= \frac{n}{2} - \frac14\left( \frac{1 - e^{2 k \pi (n+1) h i }}{1 - e^{2 k \pi h i}}  - 1 + \frac{1 - e^{-2 k \pi (n+1) h i }}{1 - e^{-2 k \pi h i}}  - 1 \right) \\
&= \frac{n+1}{2},
\end{align*}
where we used $(n+1) h = 1$, and the fact that the complex exponential is one if the argument is any integer times $2 \pi$. This implies that we have $\alpha = \sqrt{\frac{2}{n+1}}$, hence in the discrete setting, we have the following normalized eigenvectors$$v_{kj} = \phi_k(x_j) = \sqrt{\frac{2}{n+1}} \sin(k \pi j h).$$
For a general discrete function, we can write the discrete sine Fourier transform by
\begin{align*}
f(x_i) = \sum_{k=1}^n c_k \phi_k(x_i),
\end{align*}
using normalized functions, we have
\begin{align*}
c_k = \sum_{j=1}^n f(x_j) \phi_k(x_j).
\end{align*}
Note that Equation (\ref{1D-inverse}) can be written as 
\begin{align*}
(A^{-1})_{ik} = h^2 \left(\frac{n+1-k}{n+1} x_i - (x_i-x_k)_+\right) = \sum_{l=1}^n c_l \phi_l(x_i).
\end{align*}
Inserting $\phi_l(x) = \sqrt{\frac{2}{n+1}} \sin(l\pi x)$, gives
\begin{align*}
c_l =  \sum_{j=1}^n (A^{-1})_{jk} \phi_l(x_j) = \sqrt{\frac{2}{n+1}} \sum_{j=1}^n h^2 \left(\frac{n+1-k}{n+1} j - (j-k)_+\right)  \sin(l \pi j h).
\end{align*}
We use the relation
\begin{align}
s_n(z) = \sum_{k=0}^n k z^k = \sum_{k=1}^n k z^k =
\begin{cases}
\frac12 n (n+1), & \text{if } z = 1, \\[1.5ex]
\displaystyle{\frac{z (1 - (n+1) z^n + n z^{n+1})}{(1-z)^2},} & \text{if } z \ne 1.
\end{cases}
\end{align}
This relation can be demonstrated by mathematical induction or by differentiation and subsequent multiplication by $z$ of the geometric series. By writing sines and cosines as linear combinations of complex exponentials, we can use the above relation to arrive at
\begin{align}
\sum_{k=1}^n k~\sin(kx) = \frac{(n+1) \sin(nx)-n\sin((n+1)x)}{4 \sin^2(\frac{x}{2})}, \quad x \ne 2p\pi,~ p \in \mathbb{Z}, \\ \sum_{k=1}^n k~\cos(kx) = \frac{(n+1) \cos(nx)-n\cos((n+1)x) - 1}{4 \sin^2(\frac{x}{2})}, \quad x \ne 2p\pi,~ p \in \mathbb{Z}.
\end{align}
The above relation is used to write 
\begin{align}
\sum_{j=1}^n (j-k)_+ \sin(jl\pi h) =
\sum_{j=k+1}^n (j-k) \sin(j l \pi h) =
\sum_{p=1}^{n-k} p ~ \sin((p+k)l \pi h).
\end{align}
This is further worked out by
\begin{align*}
\sum_{p=1}^{n-k} p \sin((p+k)l \pi h) 
&= \sin(k l \pi h) \sum_{p=1}^{n-k} p \cos(p l \pi h) + \cos(k l \pi h) \sum_{p=1}^{n-k} p \sin(p l \pi h) \\
&= \sin(k l \pi h) \frac{(n-k+1) \cos((n-k)l \pi h) - (n-k) \cos((n-k+1) l \pi h) - 1}{4 \sin^2(\frac{l \pi h}{2})} \\
&\quad + \cos(k l \pi h) \frac{(n-k+1) \sin((n-k)l \pi h) - (n-k) \sin((n-k+1) l \pi h)}{4 \sin^2(\frac{l \pi h}{2})} \\
&= \frac{(n-k+1) \sin(n l \pi h) - (n-k)\sin(l (n+1) \pi h) - \sin(k l \pi h)}{4 \sin^2 (\frac{l \pi h}{2})} \\
&= \frac{(n-k+1) \sin(n l \pi h) - \sin(k l \pi h)}{4 \sin^2(\frac{l \pi h}{2})}.
\end{align*}
Here we used $\sin(x+y) = \sin x \cos y + \cos x \sin y$, $(n+1)h = 1$ and the fact that the sine is zero in arguments that are multiples of $\pi$. Hence, for $c_l$, we obtain
\begin{align*}
c_l &= \sqrt{\frac{2}{n+1}} h^2 \left[ \frac{n+1-k}{n+1} \frac{(n+1) \sin(n l \pi h)}{4 \sin^2(\frac{l \pi h}{2})} - \frac{(n-k+1) \sin(n l \pi h) - \sin(k l \pi h)}{4 \sin^2(\frac{l \pi h}{2})} \right] \\
&= \sqrt{\frac{2}{n+1}} h^2 \frac{\sin(k l \pi h)}{4 \sin^2(\frac{l \pi h}{2})}.
\end{align*}
Hence, this gives 
\begin{align*}
(A^{-1})_{ik} = \sum_{l = 1}^n c_l \sqrt{\frac{2}{n+1}} \sin(i l \pi h) = 
\frac{h^2}{2 (n+1)} \sum_{l = 1}^n \frac{\sin(k l \pi h)}{\sin^2(\frac{l \pi h}{2})} \sin(i l \pi h).
\end{align*}
\end{proof}
By this we have demonstrated consistence of the current result with the earlier \citep{Vermolen_2022}, simpler result for one spatial dimension with Dirichlet boundary conditions. Consistency for different boundary conditions can be approached using similar principles, although the algebra may be a little more tedious. We therefore omit this at this stage. Our new approach gives the exact inverse of the matrix represented by a discrete Fourier transform. We realize that the new approach to invert the one-dimensional Laplace matrix is much more toilsome than the earlier procedure. However, computing the solution from a polynomial matrix equation was not possible with the earlier method, whereas the current method can be used to compute the solution to this problem. For the sake of illustration, we solve the following problem
\begin{align*}
\begin{cases}
-u'' + \alpha u = 1, \text{ for } x \in (0,1), \\
u(0) = u(1) = 0.
\end{cases}
\end{align*}
The solution can be seen in Figure \ref{1d} for different values of $\alpha \ge 0$. For $\alpha > 0$, the simple formula previously found for the inverse \citep{Vermolen_2022} cannot be used since the fundamental solution is composed by hyperbolic sines and cosines and hence contains non-zero higher-order derivatives, which gives a nonzero truncation error. In \citep{Vermolen_2022}, an approximation of the inverse based on hyperbolic sines and cosines with error $\mathcal{O}(h^{3/2})$ has been derived.
\begin{figure}
\centering
\includegraphics[width=13cm]{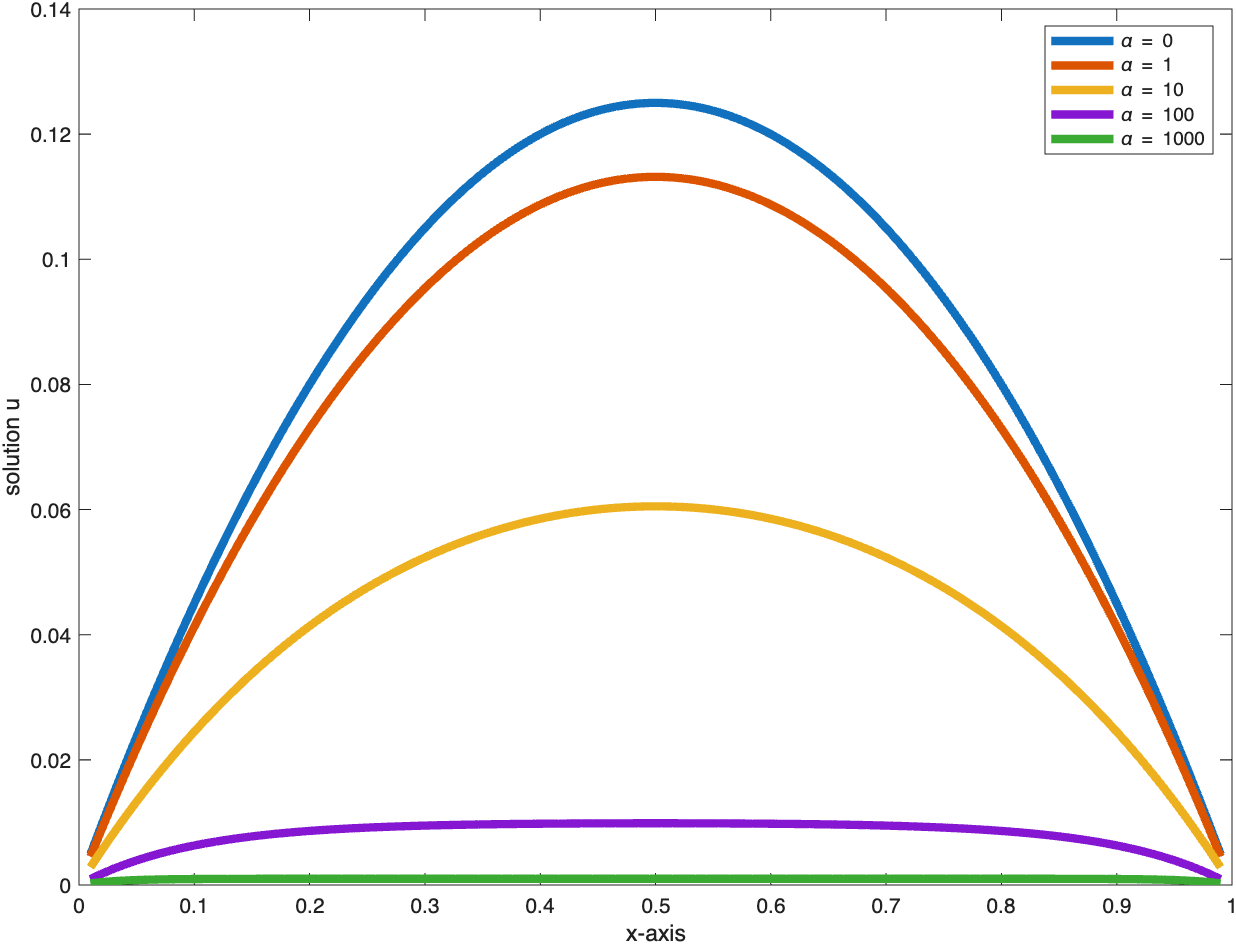}
\caption{\textit{Solutions to the equation
$-u'' + \alpha u = 1$ for different $\alpha$-values $\alpha = 0, 1, 10, 100, 1000$.}}
\label{1d}
\end{figure}
\par 
\subsection{Two dimensional Laplace on a rectangle}
Gueye \citep{Serigne} gave an exact inversion of a symmetric pentadiagonal Toeplitz matrix for the semi-analytic solution of 2D Poisson equation. For this case, no explicit relation for the inverse of the discretization matrix is known to the best of our knowledge.
We consider a unit square with $N$ unknowns at equidistant spacing $h$ in both directions. Hence there are $n = N^2$ unknowns in total. We consider the simple Laplacian with homogeneous Dirichlet conditions, which typically reads as 
\begin{align*}.
\begin{cases}
- \Delta u = f(x,y), \text{ for } (x,y) \in \Omega = (0,1)^2, \\
u|_{\partial \Omega} = 0.
\end{cases}
\end{align*}
The corresponding eigenvalue problem is given by
\begin{align*}
-\Delta \hat{\phi}_{j_1,j_2} = \hat{\lambda}_{j_1,j_2} \hat{\phi}_{j_1,j_2} =   \hat{\mu}^2_{j_1,j_2} \hat{\phi}_{j_1,j_2} = (\hat{\mu}_{j_1}^2 + \hat{\mu}_{j_2}^2) \hat{\phi}_{j_1,j_2},
\end{align*}
where $j_1, j_2 \in \mathbb{N}^{\times}=\mathbb{N} \setminus \{0\}$, and  $\hat{\mu}^2_{j_p} = \pi^2 j_p^2$, hence for the eigenvalues, we have
\begin{align*}
\hat{\lambda}_{k_1,k_2} = 
\hat{\mu}^2_{k_1,k_2} = \pi^2 (k_1^2 + k_2^2), \qquad k_j = 1,2,3,\ldots,
\end{align*}
and for the eigenfunctions we have
\begin{align*}
\hat{\phi}_{j_1,j_2}({x,y}) =
2  \sin(\pi j_1 x) \sin(\pi j_2 y).
\end{align*}
The 2D finite difference approach with constant spacing and $N$ unknowns per spatial dimension, gives the following discretization matrix (shown for $N = 4$)
\[
A = \frac{1}{h^2} \cdot
\begin{pmatrix}
  4  & -1 &  0  & -1 &  0  &  0  &  0  &  0  & 0 & ...   \\
 -1 &  4 & -1 &  0  & -1 &  0  &  0  & 0  & 0 &  ...  \\
   0  & -1 &  4 &  0  &  0  & -1 &  0  &  0  & 0 &  ...  \\
 -1 &  0  &  0  &  4 & -1 &  0  & -1 &  0  & 0 & ...   \\
   0  & -1 &  0  & -1 &  4 & -1 &  0  & -1 & 0 & ...   \\
   0  &  0  & -1 &  0  & -1 &  4 &  0  &  0  & -1 &... \\
   0  &  0  &  0  & -1 &  0  &  0  &  4 & -1 & 0 & ...   \\
   0  &  0  &  0  &  0  & -1 &  0  & -1 &  4 & -1 & ...\\
   0  &  0  &   0 &  0  &  0  & -1 &  0  & -1 &  4 & ... \\
   : & : & : & : & : & : & : & : & : & :  \\
\end{pmatrix},
\]
with the following eigenvalues for the discrete system
\begin{align}
\lambda_{k_1,k_2} =
\frac{4}{h^2} \left( \sin^2(\frac{k_1 \pi h}{2}) + \sin^2(\frac{k_2 \pi h}{2}) \right), \qquad \text{for $k_j \in \{1,\ldots,N\}$,}
\end{align}
with eigenvectors
\begin{align*}
v_{\hat{j},\hat{k}} = w_{(j_1,j_2),(k_1,k_2)} =
 2  ~ \sin(j_1 \pi k_1 h) ~ \sin(j_2 \pi k_2 h), \qquad \text{for }
 j_m, ~ k_m \in \{1,\ldots,N\},
\end{align*}
where the $j$-indexes and $k$-indexes, respectively, denote the label of the eigenvector and the entry of this eigenvector. The eigenvalues are easily determined by substitution of the eigenvectors into the system $A \underline{v} = \lambda \underline{v}$ and by similar treatment as in 1D. In the current notation, the eigenvector is a two-dimensional array (matrix). In order to transform this into a one dimensional (solution) vector, we use the following transformation
\begin{align*}
\begin{cases}
\displaystyle{\hat{j} = (j_2-1) N + j_1,} \\[1.5ex]
\displaystyle{\hat{k} = (k_2-1) N + k_1,}
\end{cases}
\text{ for } j_1,j_2,k_1,k_2 \in \{1,\ldots,N\}, \quad \text{ and } \quad  \hat{j}, \hat{k} \in \{1,\ldots n\},
\end{align*}
to make $\underline{v}$ a one dimensional array of length $n = N^2$. The inverse transformation is given by
\begin{align*}
(j_1,j_2) =
\begin{cases}
(\text{mod}(\hat{j},N),\text{trunc}(\hat{j},N)+1), & \text{ if } \text{mod}(\hat{j},N) \ne 0, \\[1.5ex]
(N,\text{trunc}(\hat{j},N)), & \text{ if } \text{mod}(\hat{j},N) = 0.
\end{cases}
\end{align*}
Further, we redefine 
\begin{align*}
\lambda_{\hat{k}} = \lambda_{k_1,k_2}, \qquad
k_1,k_2 \in \{1,\ldots,N\}, \qquad
\hat{k} \in \{1,\ldots,n\}.
\end{align*}
To solve the equation $A \underline{x} = \underline{b}$, we get
\begin{align}
{\underline x} = \sum_{j=1}^{n} \frac{1}{\lambda_j} (\underline{b},\underline{v}_j) \underline{v}_j.
\label{gen-sol1}
\end{align}
For the sake of illustration, we apply the method to the simple Laplace problem 
\begin{align*}
\begin{cases}
-\Delta u = 1, \text{ in } \Omega = (0,1)^2, \\
u|_{\partial \Omega} = 0.
\end{cases}
\end{align*}
The solution is compared to the classical use of a direct solver in Matlab, and the solution by both methods is plotted in Figure \ref{laplace2d}. No difference can be observed, and the numerical difference was of the order of machine precision. 
\begin{figure}
\centering
\includegraphics[width=15cm]{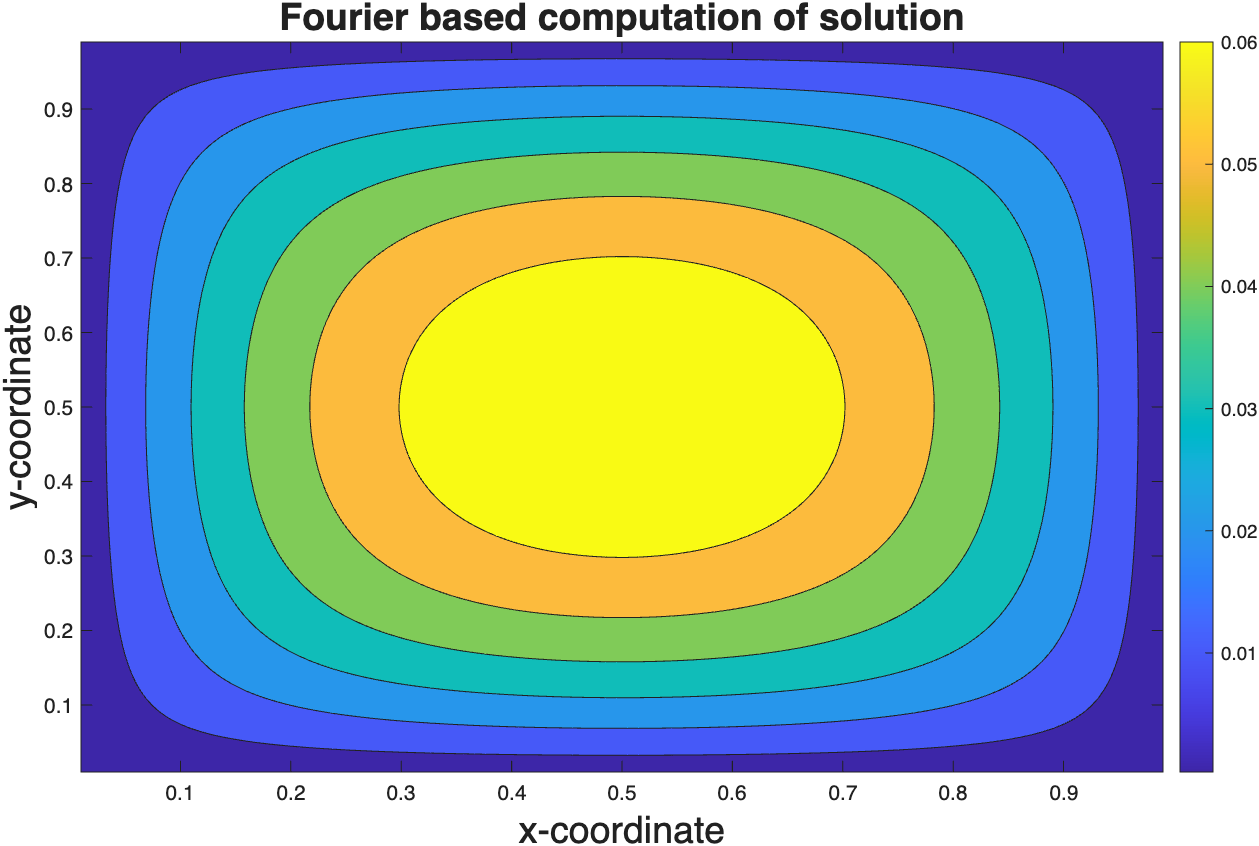}
\includegraphics[width=15cm]{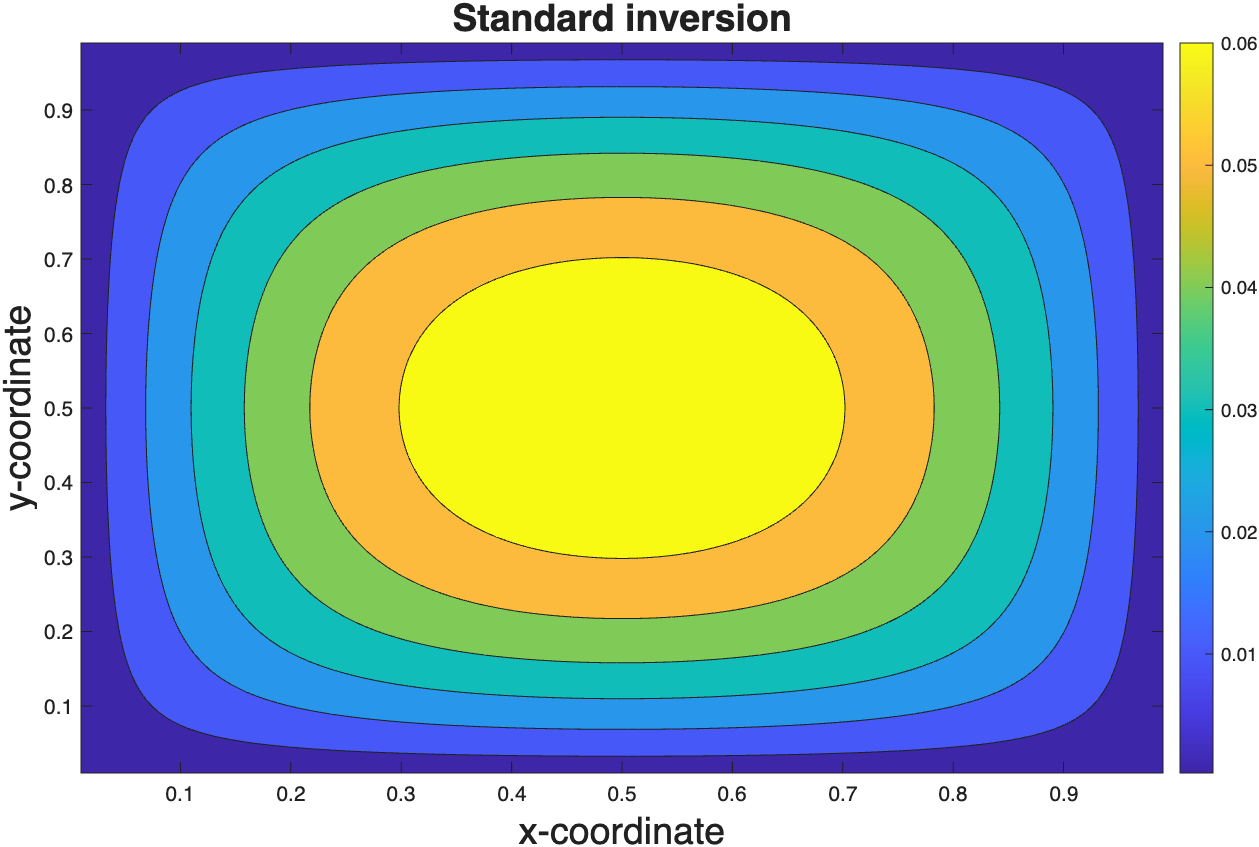}
\caption{\textit{Solutions to the Poisson equation $-\Delta u = 1$ with $u|_{\partial \Omega} = 0$. The top figure represents the solution obtained by the current procedure, the bottom figure represents the solution by classical solution methods.}}
\label{laplace2d}
\end{figure}
The same was repeated in Figure \ref{fourth2d}, where we solved the following boundary value problem:
\begin{align*}
\begin{cases}
\Delta^2 u - \Delta u + u = 1, \text{ in } \Omega = (0,1)^2, \\
u|_{\partial \Omega} = 0, ~\displaystyle{\frac{\partial u}{\partial n}|_{\partial \Omega}} = 0.
\end{cases}
\end{align*}
No difference between the two approaches has been observed with a difference in the order of machine precision.\\[2ex]
For the sake of illustration of the use of boundary conditions, we consider the case
\begin{align*}
\begin{cases}
-\Delta u = 1, \qquad \text{in } \Omega = (0,1)^2, \\
u(0,y) = u(x,0) = 0,\text{ and } \frac{\partial u}{\partial n} = 0, \text{ for } x = 1 \text{ and } y = 1.
\end{cases}
\end{align*}
The eigenvalues and normalized eigenfunctions for the continuous problem are given by
\begin{align*}
\hat{\lambda}_{k_1,k_2} = \left(\frac{\pi}{2}\right)^2 ((2k_1-1)^2 + (2k_2 - 1)^2),  \quad k_1,k_2 = 1,2,3,\ldots, \\ \\
\phi_{k_1,k_2}(x,y) = 2~ \sin\left(\frac{\pi}{2}(2k_1-1)x\right) ~ \sin\left(\frac{\pi}{2}(2k_2-1)x\right), \quad k_1,k_2 = 1,2,3,\ldots.
\end{align*}
Using nodes $x_j = j h$, $y_j = jh$, $x_n = y_n = 1$, gives $h = \frac{1}{N}$. This gives the following eigenvalues of the discretization matrix
\begin{align*}
\lambda_{k_1,k_2} = \frac{4}{h^2} \left(\sin^2(\frac{\pi}{4} (2k_1-1))+ \sin^2(\frac{\pi}{4} (2k_2-1))\right), \qquad k_1,k_2 \in \{1,\ldots,N\}.
\end{align*}
Similar procedures as in the case of fully Dirichlet conditions can be used here to get the inverse of the discretization matrix and the solution to a matrix equation $A \underline{x} = \underline{b}$.
\begin{figure}
\centering
\includegraphics[width=15cm]{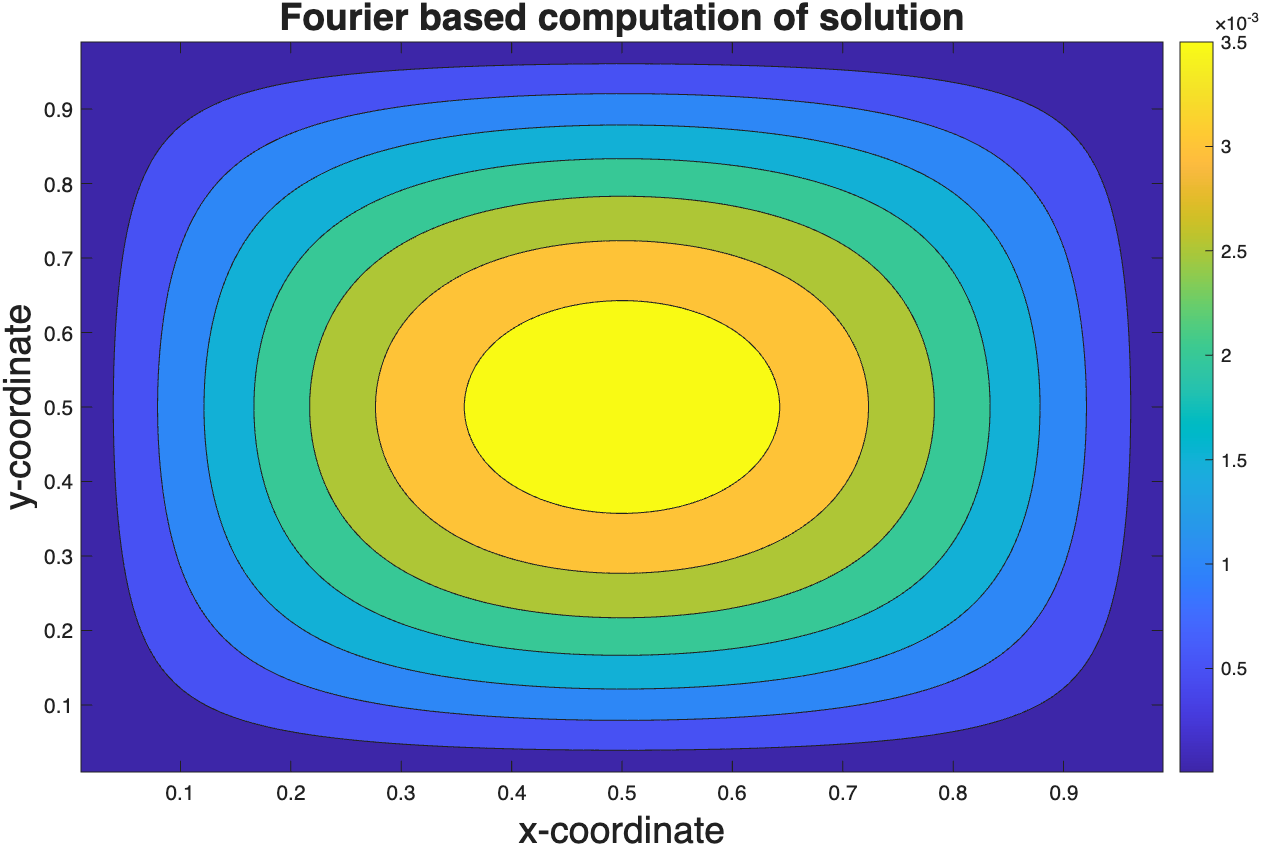}
\includegraphics[width=15cm]{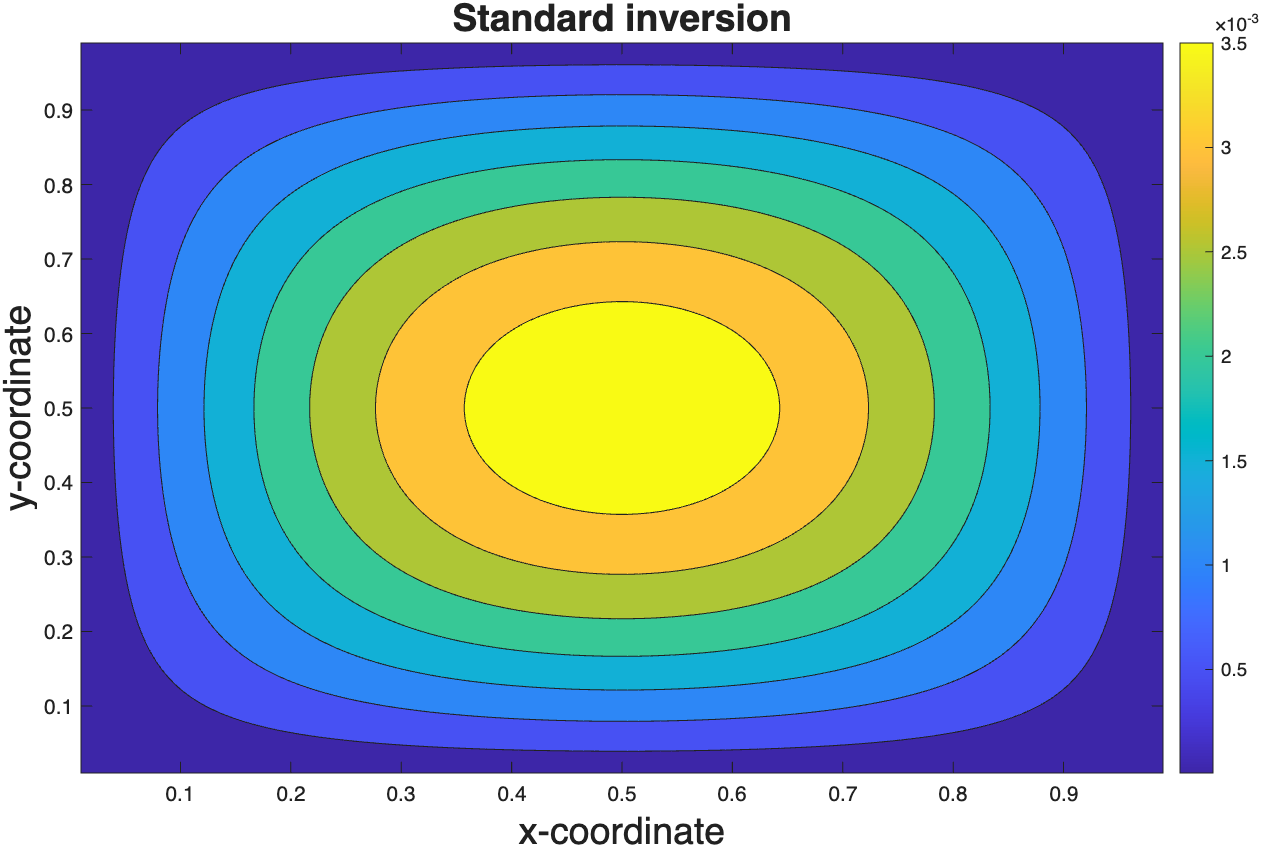}
\caption{\textit{Solutions to the fourth-order equation $\Delta^2 u - \Delta u + u = 1$ with $u|_{\partial \Omega} = 0$ and $\frac{\partial u}{\partial n}|_{\partial \Omega}=0$. The top figure represents the solution obtained by the current procedure, the bottom figure represents the solution by classical solution methods.}}
\label{fourth2d}
\end{figure}
\subsection{Higher dimensional Laplace matrix on a (hyper) cube}
For the time being, we consider a unit (hyper)cube domain in $\mathbb{R}^d$ with a Laplacian under Dirichlet boundary conditions, which reads as
\begin{align*}
\begin{cases}
- \Delta u = f({\bf{x}}), \text{ for } {\bf{x}} \in \Omega = (0,1)^d, \\
u|_{\partial \Omega} = 0.
\end{cases}
\end{align*}
The same analysis can be done for a 'hyperbeam' in $\mathbb{R}^d$ as well. The corresponding eigenvalue problem is given by
\begin{align*}
-\Delta \hat{\phi}_{j_1,\ldots,j_d} = \hat{\mu}^2_{j_1, \ldots, j_d} \hat{\phi}_{j_1,\ldots,j_p} = \sum_{p=1}^d \mu_{j_p}^2 \hat{\phi}_{j_1,\ldots,j_p},
\end{align*}
where $\mu_{j_p}^2$ represents the $j_p^{\text{th}}$ eigenvalue of the operator $-\frac{\partial^2}{\partial x_p^2}(.)$ in the $p^{\text{th}}$ coordinate direction. The eigenfunctions $\hat{\phi}_{j_1,\ldots,j_d}: \mathbb{R}^d \longrightarrow \mathbb{R}$, are given by 
\begin{align*}
\hat{\phi}_{j_1,\ldots,j_d}({\bf x}) =
\displaystyle{\prod_{p = \{1,\ldots,d\}} \phi_{p,j_p}} (x_p) = 2^{d/2} \prod_{p \in \{1,\ldots,d\}} \sin(\pi j_p x_p),
\text{where} \phi_{p,j_p}(x_p) = \sqrt{2} \sin(\pi j_p x_p),
\end{align*}

for a Dirichlet problem in a hypercube $[0,1]^d$. The eigenvalues are given by 
\begin{align*}
\hat{\lambda}_{k_1,\ldots,k_d} = 
\hat{\mu}^2_{k_1,\ldots,k_d} = \pi^2 (k_1^2 + \ldots + k_d^2), \qquad k_j \in \mathbb{N}^{\times}.
\end{align*}
The multi-dimensional finite difference approach allows the construction of eigenvectors by projecting the eigenfunctions on the meshpoints. Having $N$ unknowns per coordinate direction, this amounts to $n = N^d$ unknowns in a d-dimensional hyperbeam. This gives the following eigenvectors
\begin{align}
\lambda_{k_1,\ldots,k_d} =
\frac{4}{h^2} \sum_{p \in \{1,\ldots,d \}} \sin^2(\frac{k_p \pi h}{2}), \qquad \text{for $k_j \in \{1,\ldots,N\}$,}
\end{align}
with eigenvectors
\begin{align*}
v_{\hat{j},\hat{k}} = w_{(j_1,\ldots,j_d),(k_1,\ldots,k_d)} =
 2^{d/2} \prod_{p \in \{1,\ldots,d\}} \sin(j_p \pi k_p h), \qquad \text{for }
 j_m, ~ k_m \in \{1,\ldots,N\},
\end{align*}
where the $j$-indexes and $k$-indexes, respectively, denote the label of the eigenvector and the entry of this eigenvector. In the current notation, the eigenvector is a two-dimensional array (matrix). In order to transform this into a one dimensional (solution) vector, we use the following transformation
\begin{align*}
\begin{cases}
\displaystyle{\hat{j} = (j_d-1) N^{d-1}+ \ldots + (j_2-1) N + j_1 =
\sum_{p=2}^d (j_p-1) N^{p-1} + j_1 =
\sum_{p=1}^d (j_p \cdot N^{p-1}) - \sum_{p=1}^{d-1} N^p,} \\[1.5ex]
\displaystyle{\hat{k} = (k_d-1) N^{d-1} + \ldots + (k_2-1) N + k_1 =
\sum_{p=2}^d (k_p-1) N^{p-1} + k_1 =
\sum_{p=1}^d (k_p \cdot N^{p-1}) - \sum_{p=1}^{d-1} N^p. }
\end{cases}
\end{align*}
where we used the simplification that in the discretization we have $n$ unknowns per dimension, hence in $\mathbb{R}^d$, we have $n^d$ unknowns.
\subsection{Time Dependent Problems}
The procedure can be used to any first order time-dependent problem. However, for the sake of presentation, we consider a time-dependent diffusion problem for $u = u({\bf x},t)$, given by 
\begin{align*}
\begin{cases}
\frac{\partial u}{\partial t} - \Delta u = 0, ~ t > 0,~ {\bf x} \in \Omega, \\
u(0,{\bf x}) = u_0({\bf x}), ~ {\bf x} \in \Omega, \\
u|_{\partial \Omega} = 0, ~ t > 0.
\end{cases}
\end{align*}
After applying a spatial discretization method (semi-discretization), one arrives at
\begin{align*}
\begin{cases}
\underline{u}' + A \underline{u} = \underline{0}, ~ t > 0, \\
\underline{u}_i(0) = u_0({\bf x}_i). 
\end{cases}
\end{align*}
Here ${\bf x}_i$ represents the position of the $i$-th nodal point in the spatial discretization. For the sake of presentation, we consider the Euler backward method for the time-integration with time-step $\Delta t$. This gives
\begin{align*}
(I + \Delta t A) \underline{u}^{\tau} = \underline{u}^{\tau - 1}.
\end{align*}
Here $\tau$ represents the time-step and $\underline{u}^{\tau}$ denotes the approximation of $\underline{u}(\tau \Delta t)$.
Recursively, this becomes
\begin{align*}
(I + \Delta t A)^{\tau} \underline{u}^{\tau} = \underline{u}^0,
\end{align*}
where $\underline{u}^0 = \underline{u}(0)$.
Hence we need to invert the matrix $(I + \Delta t A)^{\tau}$. Note that this expression can be rewritten in the following polynomial form
\begin{align*}
(I + \Delta t A)^{\tau} = \sum_{i=0}^{\tau} \binom{\tau}{i} \Delta t^i A^i.
\end{align*}
It is easy to show that the eigenvectors of $A$ and $(I + \Delta t A)^{\tau}$ are the same and that the eigenvalues of $(I + \Delta t A)^{\tau}$ are given by $(1+\Delta t \lambda)^n$, where $\lambda$ is any eigenvalue of $A$. This can be written as the following closed-form expression
\begin{align*}
(I + \Delta t \lambda)^{\tau} = \sum_{i=0}^{\tau} \binom{\tau}{i} \Delta t^i \lambda^j.
\end{align*}
Hence, for the solution $\underline{u}^{\tau}$, we get the following formal closed-form expression for the inverse
\begin{align}
{\underline u}^{\tau} = 
\sum_{j=1}^n \frac{1}{(1+\Delta t \lambda_j)^{\tau}} (\underline{u}_0,\underline{v}_j) \underline{v}_j,
\label{sol-heat-eq}
\end{align}
Hence, for the inverse of $(I + \Delta t A)^{\tau}$, we get
\begin{align}
(((I + \Delta t A)^{\tau})^{-1})_{ik} = 
\sum_{j=1}^n \frac{{v}_{jk}}{n \left(1+\Delta t \lambda_j \right)^{\tau}}  ~v_{ji}. 
\end{align}
We note that this can be done similarly for other time-integration methods like the Trapezoidal (Crank-Nicolson) Rule. For the Trapezoidal time-integration method, one arrives at
\begin{align}
{\underline u}^{\tau} = 
\sum_{j=1}^n \left(\frac{1-\frac{\Delta t \lambda_j}{2}}{1+\frac{\Delta t \lambda_j}{2}}\right)^{\tau}(\underline{u}_0,\underline{v}_j) \underline{v}_j,
\end{align}
\begin{figure}
\centering
\includegraphics[width=15cm]{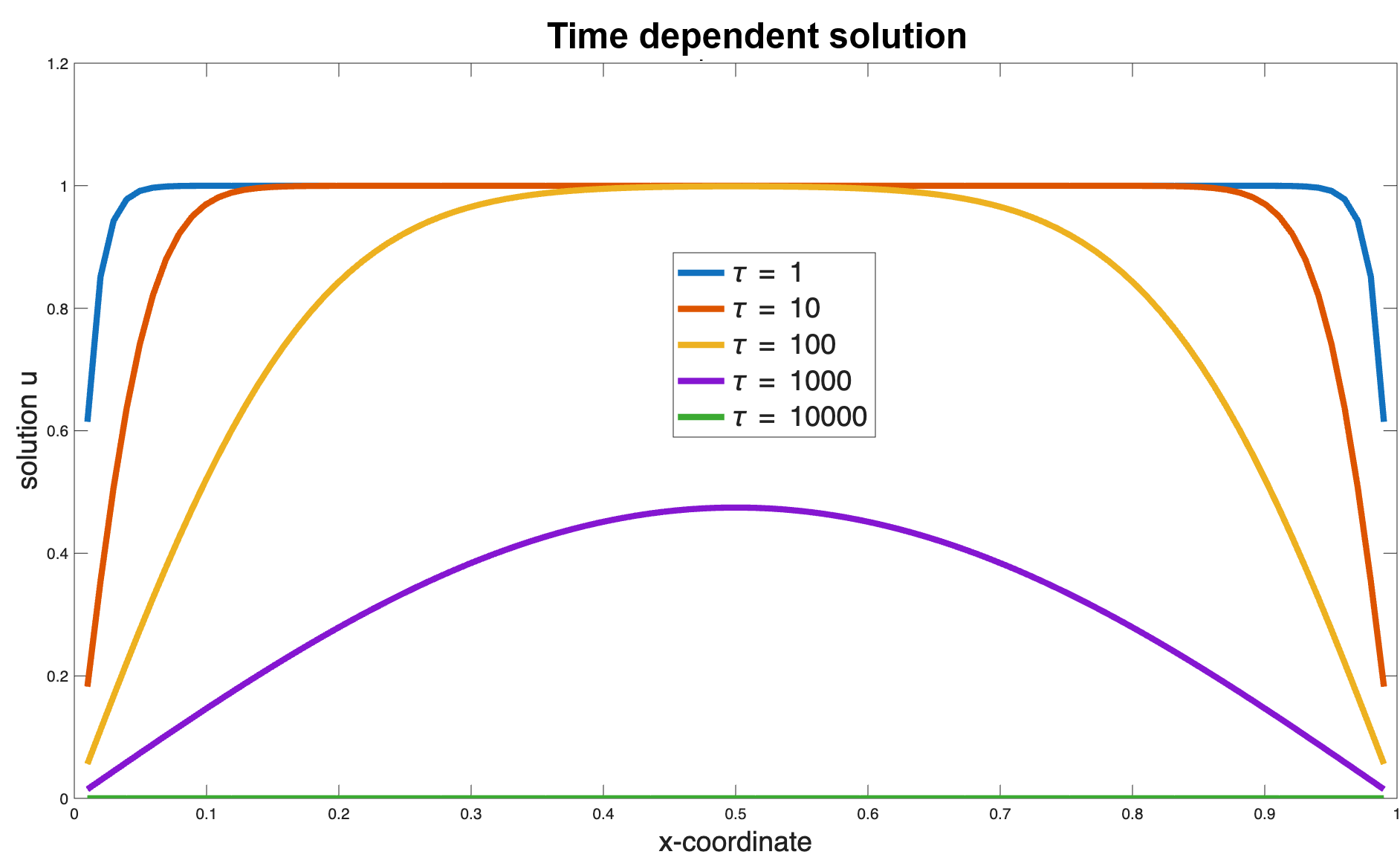}
\caption{\textit{Snapshot at consecutive times of the heat equation $u_t - u_{xx}=0$ with $u|_{\partial \Omega} = 0$ and $u = 1$ at $t = 0$ using the backward Euler time-integration method for different $\tau$-values $\tau = 1, 10, 100, 1000, 10000$.}}
\label{heat1d}
\end{figure}
This theory can be used to obtain closed-form expressions for the numerical solution to some time-dependent problems involving (polynomials) of discrete Laplace matrices. In Figure \ref{heat1d}, we show some snapshots for a one-dimensional heat equation $u_t = u_{xx}$ with homogeneous Dirichlet boundary conditions and $u = 1$ at $t = 0$. The time-step was $\Delta t = 0.001$, and a backward Euler time-integration method was used. It can be seen that the solution behaves as expected: convergence to zero due to the boundary conditions. A major advantage of the method is that the numerical solution at any time can be obtained from the initial condition, without the need of computing solutions at previous time-steps. The only operation that differs for different time-steps is the power that one has to raise for the amplification factor containing the eigenvalues of the discretization matrix. This makes the method very efficient. We have done simulations with the eigenvalue expansion method and with the ordinary finite difference method. We solved a 2D heat equation (second order PDE), given by
\begin{align*}
\begin{cases}
\frac{\partial u}{\partial t} - \Delta u = 0, \\
u({\bf x},0) = 1, \text{ on } \Omega, \\  u|_{\partial \Omega} = 0, \text{ for } t > 0.
\end{cases}
\end{align*}
and a 4th order equation in 2D, given by
\begin{align*}
\begin{cases}
\frac{\partial u}{\partial t} - \Delta u  + \Delta^2 u = 0, \\
u({\bf x},0) = 1, \text{ on } \Omega, \\  u|_{\partial \Omega} = \frac{\partial u}{\partial n}|_{\partial \Omega}= 0, \text{ for } t > 0.
\end{cases}
\end{align*}
The computation time (wall clock time) was 0.29 seconds (s) in all cases, regardless of the order of the problem and the number of time-iterations, whereas the wall clock time for the finite difference method was determined by the order of the PDE and linearly increases with the number of time-iterations, as expected (see Figure (\ref{heat1d_comptime})). This shows that the new method is very fast, in particular for large numbers of time-iterations.
\begin{figure}
\centering
\includegraphics[width=15cm]{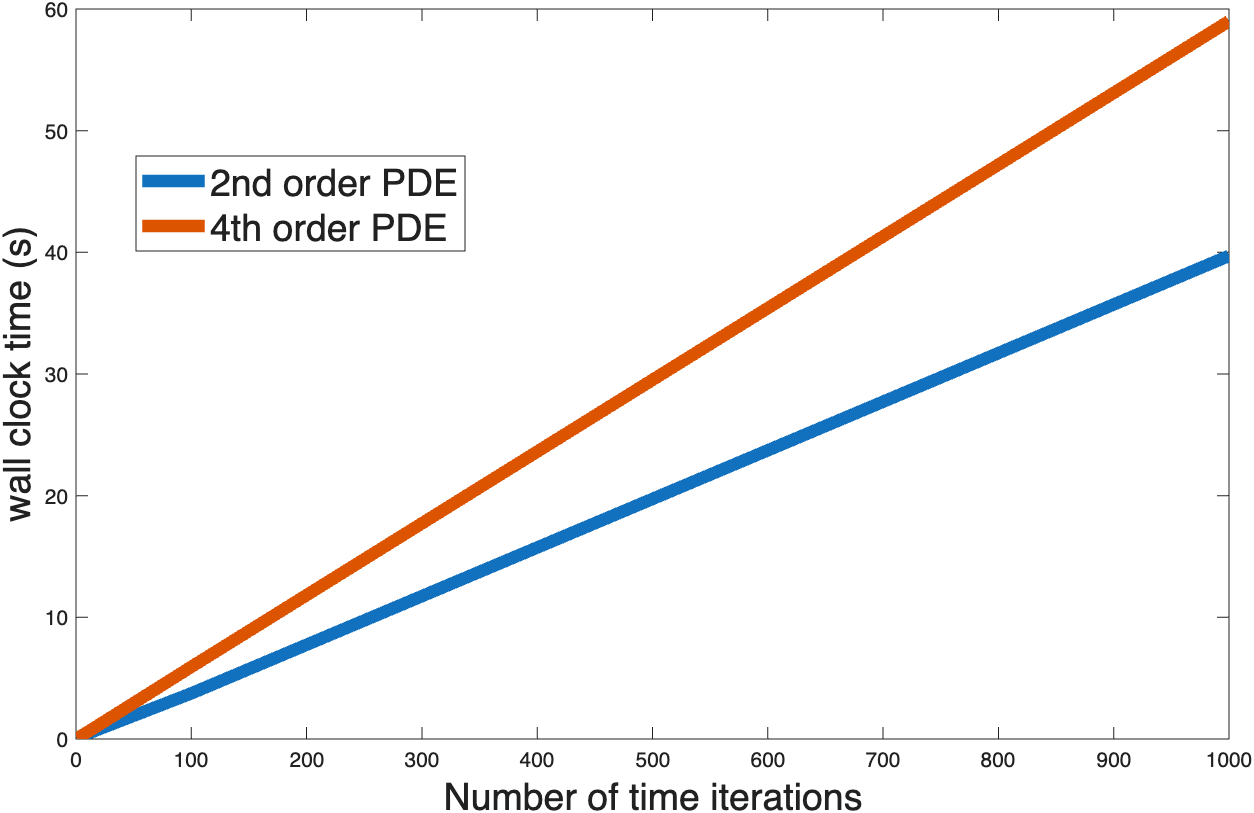}
\caption{\textit{Computation times (wall clock times) using tic-toc in Matlab for a second and fourth order PDE.}}
\label{heat1d_comptime}
\end{figure}

\section{Discussion and Conclusions}
We have designed a procedure to invert matrices that are in the polynomial space of Laplace matrices for higher dimensional problems, for all cases where the matrix polynomial is invertible. The procedure can be used to write a closed-form solution to a class of large systems of linear algebraic equations. The method is applicable to problems in (hyper) beams with a regular mesh distribution on linear operators that are based on matrix polynomials of Laplace equations. The method is based on the use of eigenvectors and eigenvalues of the Laplacian, which can be determined exactly. We are aware of the issues with Robin boundary conditions, where the determination of eigenvalues amounts to solving transcendental equations. Then the determination of the inverse requires the (numerical) solution of these transcendental equations, which gives a numerical trait. We further acknowledge that the method is elaborate and that in future studies it will be important to carry out a more detailed comparison of the efficiency of our method with classical iterative schemes such as the (preconditioned) conjugate gradient method or other iterative solution procedures. We are also aware that the problems that we can tackle with this method are idealized in the sense of simple geometries and Laplace-based problems with constant coefficients. Nevertheless it is possible to tackle polynomials of Laplace matrices, and therefore we think that our approach certainly has some theoretical value. The current method is also helpful to construct closed form expressions for higher order partial differential equations. Furthermore, time-dependent heat equations with implicit time integration on (hyper) beams have also been treated with the current method since the matrix to be inverted amounts to a polynomial of the discrete Laplace matrix. In case of a backward Euler time integration, a linear relation like $I + \Delta t A$, where $A$ represents the discrete Laplacian, needs to be inverted. For this class of problems, we observed that the computation time does not depend significantly on the number of time-steps and the order of the PDE. Whereas for the classical solving in finite difference methods, the computation time increases linearly with the number of time-iterations. In addition, the computation time also depends on the order of the PDE in case of classical solving in finite difference methods. A limiting factor is that the current method is applicable to linear problems in (hyper) beams. For these problems one can also develop closed-form expressions using separation of variables in continuous problems. However, extension to more complex geometries could be done using eigenvalue and eigenvector determination, where one only include the most pivotal eigenvalues for time integration. This matter can be investigated in future studies.\\[2ex]
Furthermore, since the application to polynomials of matrices is so straightforward, we can use the method to train a neural network with fundamental solutions (in the continuous sense with Dirac delta distributions) for powers of Laplace matrices for different spatial dimensionalities, so that the neural network can provide relatively inaccurate, but very quick solutions to real systems of linear equations with these matrix polynomials. Here, a DeepONet architecture \citep{Li_2025} could be an interesting candidate.\\[2ex]
Another application of the current method could reside in Laplace filtering, where one uses the (multidimensional) Laplace kernel to capture sharp edges and sharp transitions in the data. It detects edges by using a second-order derivative to measure the rate of change in an image. Often these derivative filters are applied to a smoothed function to avoid problems
with image noise amplification \citep{Gonzalez_2002}. Our method naturally extends to higher-dimensional manifolds (non-curved), making it particularly advantageous for applications in medical imaging, such as processing 3D MRI and CT data, and in scientific visualization, where robust multi-dimensional edge detection is essential.
\section*{Acknowledgments}
This work was supported by Research England under the Expanding Excellence in England (E3) funding stream, which was awarded to MARS: Mathematics for AI in Real-world Systems in the School of Mathematical Sciences at Lancaster University. Further, we are grateful for the financial support from the Higher Education Commission (HEC) of Pakistan in the framework of project: 1(2)/HRD/OSS-III/BATCH-3/2022/HEC/527.

\end{document}